\documentclass{article}
\usepackage[utf8]{inputenc}

\usepackage{array}    
\usepackage{caption} 
\usepackage{subcaption} 
\usepackage{hyperref}
\usepackage{latexsym}
\usepackage{amsmath,amssymb,amsthm}
\usepackage{epstopdf}
\usepackage{graphics}
\usepackage{lineno} 
\usepackage{multicol}
\usepackage{tikz}
\usetikzlibrary{shapes,arrows}  
\usetikzlibrary{positioning,chains,fit,shapes,calc}
\usepackage{graphicx}
\definecolor {processblue}{cmyk}{0.96,0,0,0}

\newtheorem{theorem}{Theorem}[section]
\newtheorem{lem}[theorem]{Lemma}

\numberwithin{equation}{section}

\begin{document}

\title{Upper Bounds on the average eccentricity of Graphs of Girth $6$ and 
$(C_4$, $C_5)$-free Graphs}

\author{Alex Alochukwu, Peter Dankelmann (University of Johannesburg)}

\maketitle

\begin{abstract}
Let $G$ be a finite, connected graph. The eccentricity of a vertex $v$ of 
$G$ is the distance from $v$ to a vertex farthest from $v$. 
The average eccentricity of $G$ is the arithmetic mean of the eccentricities
of the vertices of $G$.
We show that the average eccentricity of a connected graph $G$ of girth at 
least six is at most $\frac{9}{2} \lceil \frac{n}{2\delta^2 - 2\delta+2} \rceil + 7$, 
where $n$ is the order of $G$ and $\delta$ its minimum degree. We construct 
graphs that show that whenever $\delta-1$ is a prime power, then this bound is 
sharp apart from an additive constant. For graphs containing a vertex of large
degree we give an improved bound. 
We further show that if the girth condition on $G$ is relaxed to $G$ having 
neither a $4$-cycle nor a $5$-cycle as a subgraph, then similar and only slightly
weaker bounds hold. 
\end{abstract}
MSC-class: 05C12

\section{Introduction}

Let $G$ be a connected graph. The {\em eccentricity} $e(v)$ of a vertex $v$ is the
distance from $v$ to a vertex farthest from $v$,  i.e., 
$e_G(v) = \max_{w \in V(G)}d_G(v,w)$, where $V(G)$ denotes the vertex set of 
$G$ and $d_G(v,w)$ is the usual distance between $v$ and $w$. 
The {\em average eccentricity}
${\rm avec}(G)$ of $G$ is defined as the arithmetic mean of the eccentricities of
its vertices, i.e., ${\rm avec}(G) = \frac{1}{n} \sum_{v \in V(G)} e_G(v)$, where 
$n$ is the order of $G$.  
The average eccentricity was introduced under the name {\em eccentric mean} by 
Buckley and Harary \cite{BucHar1990}, but it attracted major attention only after 
its first systematic study in \cite{DanGodSwa2004}. One of the basic results in
this paper determined the maximum average distance of a connected graph of given order:

\begin{theorem} {\rm  \cite{DanGodSwa2004}}  \label{theo:path-maximises-avec}
If $G$ is a connected graph of order $n$, then 
\[
 {\rm avec}(G) \leq \frac{1}{n} \left \lfloor \frac{3n^2}{4} -\frac{n}{2} \right \rfloor, 
\] 
with equality if and only if $G$ is a path.  
\end{theorem}

Several bounds on the average eccentricity have been found since. For example for 
graphs of given order and size \cite{AliDanMorMukSwaVet2018, TanZho2012}, 
and for maximal planar graphs \cite{AliDanMorMukSwaVet2018}. Several relations 
between the average eccentricity 
and other graph parameters, for example 
independence number \cite{DanMuk2014, DanOsa2019, Ili2012}, domination number 
\cite{DanMuk2014, DanOsa2019, DuIli2013, DuIli2016, Du2017}, clique number 
\cite{DasMadCanCev2017, Ili2012}, chromatic number \cite{TanWes2019}, 
proximity \cite{MaWuZha2012} 
and Wiener index \cite{DarAliKlaDas2018} have been explored.
Bounds on the average eccentricity of the strong product of graphs were given
in \cite{CasDan2019}.

The natural question if the bound in Theorem \ref{theo:path-maximises-avec}
can be improved for graphs whose minimum degree is greater than $1$ was answered
in the affirmative in \cite{DanGodSwa2004}, where it was shown that 
if $G$ is a graph of order $n$ and minimum degree $\delta$, then 
\begin{equation}    \label{eq:bound-avec-given-n-delta}
{\rm avec}(G) \leq \frac{9n}{4(\delta+1)} + \frac{15}{4}, 
\end{equation}
and this inequality is best possible apart from a small additive constant. 
Further results relating the average eccentricity of a graph to its vertex
degrees are known. Bounds on the average eccentricity of trees of given order and 
maximum degree were given in \cite{Ili2012}. Trees with given degree sequence 
that minimise or maximise the average eccentricity were determined in \cite{SmiSzeWan2016}. 
For relations between average eccentricity and Randi\'{c} index see 
\cite{LiaLiu2012}.  
An upper bound on the average eccentricity in terms of order, size and
first Zagreb index was
given in \cite{DasMadCanCev2017}.

It was observed in \cite{DanMukOsaRod2019} that the upper bound 
\eqref{eq:bound-avec-given-n-delta} can be improved for triangle-free graphs
and for graphs not containing four-cycles. 
The aim of this paper is to further pursue the idea of improving 
\eqref{eq:bound-avec-given-n-delta} for graphs not containing certain subgraphs. 
In this paper we give upper bounds on the average eccentricity 
of graphs of girth at least $6$, and of  graphs containing neither $4$-cycles 
nor $5$-cycles, in terms of order,  minimum degree and maximum degree.

The notation we use is as follows. 
We denote the vertex set and edge set of a graph $G$ by $V(G)$ and $E(G)$,
respectively, and $n(G)$ stands for the order of $G$, i.e., for 
the number of vertices of $G$. 
By ${\rm deg}_G(v)$ we mean the degree of $v$, i.e., the number of vertices adjacent 
to $v$. The largest of the eccentricities of the vertices of $G$ is called the 
{\em diameter} of $G$ and denoted by ${\rm diam}(G)$..  

For $k\in \mathbb{Z}$, we denote the set of vertices at distance exactly $k$
and at most $k$ from a vertex $v$ by $N_k(v)$ and $N_{\leq k}(v)$, respectively.
If $uv$ is an edge of $G$, then $N_{\leq k}(uv)$ is the set 
$N_{\leq k}(u) \cup N_{\leq k}(v)$.
The $k$-th power of $G$, denoted by $G^k$, is the graph with the same vertex set
as $G$ in which two vertices are adjacent if their distance is not more than $k$. 

The {\em line graph} of a graph $G$ is the graph $L$ whose vertex set is $E(G)$,
with two vertices of $L$ being adjacent in $L$ if, as edges of $G$, they share 
a vertex.

A {\em matching} of $G$ is a set of edges in which no two edges share a vertex.
The vertex set $V(M)$ of a matching $M$ is the 
set of vertices incident with an edge in $M$. The distance $d_G(e_1, e_2)$ between 
two edges $e_1$ and $e_2$ is the smallest of the distances between a vertex
incident with $e_1$ and a vertex incident with $e_2$. (Note that in general this
is not equal to the distance in the line graph of $G$.) If $M$ is a set of edges,
then the distance $d(e,M)$ between an edge $e$ and $M$ is the smallest of the distances
between $e$ and the edges in $M$. 

If $A\subseteq V(G)$, then we write $G[A]$ for the sugbgraph of $G$ induced by $A$.

By $C_n$ we mean the cycle on $n$ vertices. We say a graph is $C_k$-free if it
does not contain $C_k$ as a (not necessarily induced) subgraph. A graph is
$(C_4,C_5)$-free if it contains neither $C_4$ nor $C_5$ as a subgraph. 
The girth of a graph $G$ is the length of a smallest cycle of $G$.

\section{Preliminary results}

In this section we present some results which will be needed for the proof
or our main theorems. 
  
If $v$ and $w$ are two adjacent vertices of a graph of girth at least $6$, 
then the sets of vertices at distance at most two from $v$ 
or $w$, respectively, in $G-vw$ are disjoint if $G$ has 
girth at least $6$, hence we have the following well-known 
result (see for example \cite{AloDan-manu}).

\begin{lem}[\cite{AloDan-manu}]  \label{la:neighbourhood-in-C4-free-bipartite}
Let $G$ be a  graph of girth at least $6$ and minimum degree $\delta$.  
If $v$ and $w$ are adjacent vertices of $G$, then 
$$|N_{\leq 2}(vw)| 
\geq 2(\delta^2 - \delta  + 1).  $$
\end{lem}

It was shown in \cite{AloDan-manu} that if we relax the girth  condition 
to $G$ having
neither $4$-cycles nor $5$-cycles (so triangles are permitted), 
then a slightly weaker bound on $|N_{\leq 2}(u) \cup N_{\leq 2}(w)|$ holds.

\begin{lem}[\cite{AloDan-manu}]  \label{lowC4C5}
Let $G$ be a $(C_{4}, C_{5})$-free graph with minimum degree $\delta \geq 3$. If $v$ and 
$w$ are adjacent vertices of $G$, then 
\[ |N_{\leq 2}(vw)| 
     \geq \left\{ \begin{array}{cc} 
      2\delta^{2} - 5\delta + 5  & \textrm{if $\delta$ is even}, \\  
      2\delta^{2} - 5\delta + 7 & \textrm{if $\delta$ is odd}. 
      \end{array} \right.        \] 
\end{lem}

We also require bounds on the number of vertices within distance three of a vertex
of large degree. 

\begin{lem}[\cite{AloDan-manu}]  \label{la:N3}  
Let $G$ be a graph of girth $6$, minimum degree $\delta \geq 3$ and maximum degree
$\Delta$. If $v$ is 
a vertex of degree $\Delta$, then 
\[ |N_{\leq 3}(v)| \geq \Delta \delta + (\delta-1)\sqrt{\Delta(\delta-2)} + \frac{3}{2}. \] 
\end{lem}

\begin{lem}[\cite{AloDan-manu}]  \label{la:N3-for-C4-C5-free}  
Let $G$ be a $(C_4, C_5)$-free graph of minimum degree $\delta \geq 3$ and maximum degree
$\Delta$. If $v$ is 
a vertex of degree $\Delta$, then 
\[ |N_{\leq 3}(v)| \geq \Delta (\delta-1) + (\delta-2)\sqrt{\Delta(\delta-3)} + \frac{3}{2}. \] 
\end{lem}

Let $G$ be a connected graph with a weight function 
$c: V (G) \rightarrow \mathbb{R}^{\geq 0}$. Then the eccentricity of $G$ with respect 
to $c$ is defined by
\[ EX_c(G) = \sum_{v\in V(G)} c(v) e_G(v). \]
If the total weight of the vertices of $G$ is strictly greater than $0$, we define 
the average eccentricity of $G$ with respect to $c$ by
\[   {\rm avec}_c(G) = \frac{\sum_{v\in V(G)} c(v) e_G(v)}{\sum_{v\in V(G)} c(v)}.   \]
We usually  denote the total weight of the vertices of $G$ by $N$. 
Hence, if $N>0$, we have ${\rm avec}_c(G) = \frac{EX_c(G)}{N}$.

\begin{lem}[\cite{DanGodSwa2004}] \label{la:weighted-avec-maximise-by-path}
Let $G$ be a connected, weighted graph with a weight function 
$c; V(G) \rightarrow \mathbb{R}^{\geq 0}$. Let 
$N = \sum_{v\in V(G)} c(v)$. If $c(v) \geq 1$ for all $v\in V(G)$, then
\[ {\rm avec}_c(G) \leq {\rm avec}(P_{\lceil N \rceil}). \] 
\end{lem}

\section{Bounds in terms of order and minimum degree}

In this section we present the first two of our main results: upper bounds on the average eccentricity 
of graphs of girth at least six and of $(C_4, C_5)$-free graphs in terms of
order and minimum degree. The basic proof strategy follows that used in \cite{DanMukOsaRod2019}.

\begin{theorem} \label{theo:avec-mindegree-girth6}
Let $G$ be a connected graph of order $n$, minimum degree $\delta \geq 3$ and 
girth at least $6$. Then
\[ {\rm avec}(G) \leq 
  \frac{9}{2} \Big\lceil \frac{n}{\delta^*} \Big\rceil + 8, \]
where $\delta^* = 2\delta^2-2\delta+2$.     
\end{theorem}

\begin{proof}
We first find a matching $M$ of $G$ as follows. 
Choose an arbitrary edge $e_1$ of $G$ and let $M= \{e_1\}$. 
If there exists an edge $e_2$ of $G$ at distance exactly $5$ from $M$, then let 
$M=\{e_1, e_2\}$.  If there exists an edge $e_3$ at distance $5$ from $M$ then
let $M=\{e_1, e_2, e_3\}$. 
Repeat this step, i.e., successively add edges at distance $5$ from $M$  
until, after $k$ steps say, each edge of $G$ is within distance $4$ of $M$.  
Let $M=\{e_1, e_2,\ldots, e_k\}$. 

The sets $N_{\leq 2}(e_i)$ are pairwise disjoint for $i=1,2,\ldots,k$. 
For $i=1,2,\ldots, k$ let $T(e_i)$ be a spanning tree of $N_{\leq 2}(e_i)$ that 
contains $e_i$ and preserves the distances to $e_i$. Since the sets 
$N_{\leq 2}(e_i)$ are pairwise disjoint, the trees $T(e_i)$ are vertex disjoint,
so the union $\bigcup_{i=1}^k T(e_i)$ forms a subforest $T_1$ of $G$.  
It follows from the construction of $M$ that for every $i \in \{2,3,\ldots,k\}$
there exists an edge $f_i$ in $G$ joining a vertex in $T(e_i)$ to a vertex in $T(e_j)$
for some $j$ with $1 \leq j <i$.  Hence $T_2:=T_1 + \{f_2, f_3,\ldots, f_k\}$ is a subtree
of $G$. Now every vertex of $G$ is within distance $5$ from some vertex of $V(M)$. 
Hence we can extend $T_2$ to a spanning tree $T$ of $G$ that preserves the distances 
to a nearest vertex in $V(M)$. 
Since the average eccentricity of any spanning tree of $G$ is not less than 
the average eccentricity of $G$, it suffices to show that
\begin{equation} \label{eq:bound-on-avec-T} 
{\rm avec}(T) \leq 
  \frac{9}{2} \Big\lceil \frac{n}{\delta^*} \Big\rceil + 8. 
\end{equation}  
For every vertex $u \in V (T)$ let $u_M$ be a vertex in $V(M)$ closest to $u$ in $T$. 
The tree $T$ can be thought of as a weighted tree, where each vertex has weight exactly $1$. 
Informally speaking, we now move the weight of every vertex to the closest vertex in 
$V(M)$. More precisely, we define a
weight function $c: V(T) \rightarrow \mathbb{R}^{\geq 0}$ by 
\[ c(v) = | \{u \in V(T) \ | \ u_M = v\}|. \]
Since $d_T(x, x_M) \leq 5$ for all $x\in V(G)$, we have
\begin{eqnarray}
|{\rm avec}_c(T) - {\rm avec}(T)| 
   & = & \big| \frac{1}{n} \sum_{u\in V(T)} c(u) e_T(u) - \frac{1}{n}\sum_{v\in V(T)} e_T(v) \big| 
                  \nonumber \\
   & = & \big| \frac{1}{n}\sum_{v\in V(T)} e_T(v_M) - \frac{1}{n}\sum_{v\in V(T)} e_T(v) \big| 
       \nonumber \\
   & \leq & \frac{1}{n} \sum_{v\in V(T)} \big| e_T(v_M) - e_T(v) \big| 
        \nonumber \\
   &\leq & \frac{1}{n}\sum_{v \in V(T)} d_T(v_M,v)  \nonumber \\
   & \leq & 5.  \label{eq:avec(T)-vs-avec_c(T)}
\end{eqnarray}
Note that $c(u) = 0$ if $u \notin V(M)$ and
$\sum_{v\in V(G)} c(v) =n$, where $n$ is the order of $G$.
We consider the line graph $L$ of $T$ and define a new weight function 
$\overline{c}$ on $V(L) = E(T)$ by
\[ \overline{c}(uv) = \left\{ \begin{array}{cc}
 c(u) + c(v) & \textrm{if $uv\in M$,} \\
 0 & \textrm{if $uv \notin M$}. 
\end{array} \right. \]
Let $uv \in M$. For each vertex $x \in N_{\leq 2}(uv)$, we have 
$x_M \in \{u, v\}$. Hence, by Lemma \ref{la:neighbourhood-in-C4-free-bipartite} 
it follows that for all $uv \in M$,  
\begin{equation}   \label{eq:lower-bound-on-N2(e)-girth6}
\overline{c}(uv) = c(u) + c(v) \geq |N_{\leq 2}(uv)|
                  \geq \delta^*. 
\end{equation}
We now bound the difference between ${\rm avec}_c(T)$ and ${\rm avec}_{\overline{c}}(L)$.  
If $x$ and $y$ are vertices of $T$, and $e_x, e_y$ are edges of $T$ 
incident with $x$ and $y$, respectively, then it is easy to prove that  
$d_T(x, y) \leq  d_L(e_x, e_y) + 1$ and consequently 
$e_T(x) \leq e_L(e_x)+1$. Since the weight of $c$ is concentrated entirely 
in the vertices in $V(M)$, we have 
\begin{eqnarray*}
\sum_{v \in V(T)} c(v) e_T(v) 
 & = & \sum_{uv \in M} c(u)e_T(u) + c(v) e_T(v) \\
 & \leq & \sum_{uv \in M} \overline{c}(uv) (e_L(uv)+1) \\
 & = & \big(\sum_{uv \in M} \overline{c}(uv) e_L(uv) \big) + n. 
\end{eqnarray*}
Division by $n$ now yields 
\begin{equation} \label{eq:avecT-vs-avecL}
{\rm avec}_c(T) \leq {\rm avec}_{\overline{c}}(L) + 1. 
\end{equation}
Now, if the distance $d_T(e_i, e_j)$ between two matching edges $e_i, e_j \in M$ 
equals five, then $d_L(e_1, e_2) \leq 6$. By the construction of $M$, every 
edge $e_i \in M$ with
$i>1$  is thus adjacent in $L^6$ to an edge $e_j \in M$ with $j<i$. 
It follows that $L^6[M]$ is connected. Moreover,  we have for all pairs $e, f \in  M$
that  
\[ d_L(e,f) \leq 6d_{L^6[M]}(e,f). \] 
Now, for every edge $e$ of $T$ there exists an edge $f \in M$ such that 
$d_L(e, f) \leq 5$. It follows that for every $f \in M$ we have 
\[
e_L(f) \leq 6e_{L^6[M]}(f) + 5,
\]
and thus
\begin{equation}  \label{eq:avec(L)-vs-avec(L6)}
{\rm avec}_{\overline{c}}(L) \leq    6 {\rm avec}_{\overline{c}}(L^6[M])  +5.
\end{equation}
To normalise the weights of the vertices of $L^6[M]$, we now define the new weight
function $\overline{c}'$ by 
$\overline{c}'(e) = \frac{\overline{c}(e)}{\delta^*}$ 
for all $e\in M$. Clearly, 
\begin{equation}   \label{eq:avec-equal-for-both-weights}
{\rm avec}_{\overline{c}'}(G) 
  = \frac{\sum_{v \in V(T)} \overline{c}'(v) e_{L^[M}}{\sum_{v \in V(T)} \overline{c}'(v)} 
  = \frac{\sum_{v \in V(T)} \overline{c}(v) e_{L^[M}}{\sum_{v \in V(T)} \overline{c}(v)}  
  = {\rm avec}_{\overline{c}}(G)
\end{equation}
Observe that $\overline{c}'(e) \geq 1$ for all $e\in M$ 
by \eqref{eq:lower-bound-on-N2(e)-girth6} and
that $\sum_{v \in V(T)} \overline{c}'(v) = \frac{n}{\delta^*}$. 
We thus have by Lemma \ref{la:weighted-avec-maximise-by-path} 
\begin{equation}   \label{eq:bound-on-avec-L6[m]}
{\rm ave}_{\overline{c}'}(L^6[M]) 
     \leq \frac{3}{4} \Big\lceil \frac{n}{\delta^*} \Big\rceil  - \frac{1}{2}.
\end{equation}
From \eqref{eq:avec(T)-vs-avec_c(T)}, \eqref{eq:avecT-vs-avecL}, 
\eqref{eq:avec(L)-vs-avec(L6)}, \eqref{eq:avec-equal-for-both-weights} and
\eqref{eq:bound-on-avec-L6[m]} we obtain
\begin{eqnarray*}
{\rm avec}(T) & \leq & {\rm avec}_c(T)+5 \\
  & \leq & {\rm avec}_{\overline{c}}(L) + 6 \\
  & \leq & 6 \; {\rm avec}_{\overline{c}}(L^6[M]) + 11 \\
  & \leq  &  6\Big( \frac{3}{4} \Big\lceil \frac{n}{\delta^*} \Big\rceil  
    - \frac{1}{2} \Big) + 11 \\ 
  & = &  \frac{9}{2} \Big\lceil \frac{n}{\delta^*} \Big\rceil + 8,
\end{eqnarray*}
which is \eqref{eq:bound-on-avec-T}, as desired. 
\end{proof}

We now show that the bound in 
Theorem \ref{theo:avec-mindegree-girth6} is sharp apart from an additive
constant whenever $\delta-1$ is a prime power. This holds even if we restrict ourselves to a  
subclass of graphs of girth at least six, to $C_4$-free bipartite graphs.

\begin{theorem}\label{theo:sharpness-example-for-bip-C4-free-avec-mindegree}
Let $\delta \in \mathbb{N}$ such that  $\delta - 1 $ is a prime power. 
Then there exists an infinite family of bipartite $C_4$-free graphs $G$ of order $n$ and minimum degree $\delta$ such that 
\[
  {\rm avec}(G) \geq  \frac{9n}{2\delta^*} - 5. 
\]
where $\delta^* = 2\delta^2 - 2\delta + 2$.
\end{theorem}

\begin{proof}
Given $\delta$, let $q=\delta-1$. Then $q$ is a prime power. Our 
construction is based on the graph $H_q$ 
first constructed by Reimann \cite{Rei1958}.
Let $GF(q)$ be the finite field of order $q$. 
Consider the $3$-dimensional vector space $GF(q)^3$, i.e., the set of all triples of
elements of $GF(q)^3$. For $i=1,2$ let $V_i$ be the set of all $i$-dimensional subspaces
of $GF(q)^3$. Now $H_q$ is defined as the bipartite graph with partite sets 
$V_1$ and $V_2$, where two vertices $v_1\in V_1$ and $v_2 \in V_2$ are adjacent if
and only if $v_1$ is a subspace of $v_2$. 
It is easy to verify that $H_{q}$ has 
$2(q^2+q+1)$ vertices, has diameter three, is $(q+1)$-regular, and that 
$H_{q}$ does not contain any $4$-cycles. 

Let $\ell \in \mathbb{N}$ with $\ell$ even, and let $uv$ be an edge of $H_{q}$. 
Let $H^{1}$ and $H^{\ell}$ be disjoint copies of $H_{q}$, and let 
$H^{2}, H^{3}, \ldots , H^{\ell-1}$ be disjoint copies of $H_{q}-uv$. 
Let $G_{\delta,\ell}$ be the graph obtained 
from the union of $H^{1}, H^{2}, \ldots H^{\ell}$ 
by adding the edges $v^{(t)}u^{(t+1)}$ for every 
$t \in \{1,2,\ldots,\ell-1\}$ where $u^{(t)}$ and $v^{(t)}$ are the vertices of 
$H_{t}$ corresponding to the  vertices $u$ and $v$, respectively, of $H_{q}$. 
Clearly, $G_{\delta,\ell}$ is bipartite and $C_4$-free, so its girth is at
least six. Its minimum degree is $\delta$. 
Since $\delta=q+1$, the order $n$ of $G_{\delta, \ell}$ is 
\[
n = 2\ell (q^{2} + q + 1) 
= 2\ell (\delta^{2} - \delta + 1) 
= \ell \delta^*.
\]
In order to bound the average eccentricity of $G_{\delta,\ell}$ from below, 
choose vertices $u^*$ of $H^1$ and $v^*$ of $H^{\ell}$ with 
$d(u^*,v^1)=d(u^{\ell},v^*)=3$. 
Since $H_{q}$ has girth at least $6$, the distance between
$u^{(i)}$ and $v^{(i)}$ in $H^i$ is at least $5$ for $i=2,3,\ldots,\ell-1$. It is easy 
to verify that in fact ${\rm diam}(H^i)=5$ for $i=2,3,\ldots,\ell-1$. Hence  
${\rm diam}(G^{*}_{\delta,\ell}) =d(u^*,v^*) =  6\ell - 5 
= \frac{3n}{\delta^{2} - \delta + 1} -  5$.
If $w \in V(H^i)$, then 
$ e(w) = d(w,v^*) \geq d(v^i,v^*) = 6(\ell-i)-2$ if $i \leq \frac{\ell}{2}$, and 
$ e(w) = d(w,u^*) \geq d(u^i,v^*) = 6(i-1)-2$ if $i > \frac{\ell}{2}$.
Hence
\begin{eqnarray*}
EX(G_{\delta,\ell}) & = & \sum_{i=1}^{\ell/2} \sum_{w\in V(H^i)} e(w) 
    + \sum_{i=\ell/2 +1}^{\ell} \sum_{w\in V(H^i)} e(w) \\
    & \geq & \sum_{i=1}^{\ell/2} \delta^* \big[ 6(\ell-i)-2 \big] 
             + \sum_{i=\ell/2 + 1}^{\ell} \delta^* \big[ 6(i-1)-2 \big] \\
   & = & {\delta}^* (\frac{9}{2}\ell^2-5\ell).
\end{eqnarray*}
Since $n= \ell\delta^*$, division by $n$ yields that
\[ {\rm avec}(G_{\delta,\ell}) \geq   
    \frac{\delta^* (\frac{9}{2}\ell^2-5\ell)}{\ell \delta^* }
    = \frac{9}{2}\ell - 5
    =\frac{9n}{2\delta^*} - 5, \]
as desired.     
\end{proof} 

If we relax the condition on $G$ to have girth
at least six to $G$ being $(C_4, C_5)$-free, we obtain a bound very similar to
Theorems \ref{theo:avec-mindegree-girth6}. 
We omit the proof as it is almost identical to that of 
Theorems \ref{theo:avec-mindegree-girth6}.

\begin{theorem} \label{theo:avec-mindegree-C4-C5-free}
Let $G$ be a connected $(C_4, C_5)$-free graph of order $n$ and minimum 
degree $\delta \geq 3$. Then
\[ {\rm avec}(G) \leq 
  \frac{9}{2} \Big\lceil \frac{n}{\delta^\circ } \Big\rceil + 8, \]
where $\delta^\circ = 2\delta^2-5\delta+5$ if $\delta$ is even, and 
$\delta^\circ = 2\delta^2-5\delta+7$ if $\delta$ is odd.      
\end{theorem}

We do not know if the bound in Theorem \ref{theo:avec-mindegree-C4-C5-free} 
is sharp. But since $\lim_{\delta \rightarrow \infty} \frac{\delta^*}{\delta^{\circ}} = 1$,
it is clear from 
Theorem \ref{theo:sharpness-example-for-bip-C4-free-avec-mindegree} that for 
large $\delta$ the coefficient of $n$ in the bound is close to being optimal.

\section{Bounds in terms of order, minimum degree and maximum degree}

We now show that the bound in Theorem \ref{theo:avec-mindegree-girth6} can 
be improved if $G$ contains a 
vertex of large degree. The proof of this bound follows broadly that of
Theorem \ref{theo:avec-mindegree-girth6}, and also borrows ideas from
\cite{DanOsa-manu}, but several modifications and 
additional arguments are required. 

\begin{theorem} \label{theo:avec-mindegree-maxdegree-girth6}
Let $G$ be a connected graph of order $n$, minimum degree $\delta \geq 3$,
maximum degree at least $\Delta$ and 
girth at least $6$. Then 
\[ {\rm avec}(G) \leq \frac{n-\Delta^*}{2\delta^*} \,  \frac{9n+3\Delta^*}{n} + 21, \]
where $\delta^*=2(\delta^2-\delta+1)$ and 
$\Delta^*= \Delta \delta + (\delta-1)\sqrt{\Delta(\delta-2)} +\frac{3}{2}$.  
\end{theorem}

\begin{proof}
Let $v_1$ be a vertex of degree $\Delta$ and let $e_1$ be an edge incident with $v_1$. We 
first find a matching $M$ of $G$ as follows. Let $M= \{e_1\}$. 
If there exists an edge $e_2$ with $d_G(e_1, e_2)=6$, add $e_2$ to $M$. 
Assume that $M=\{e_1, e_2,\ldots,e_{i-1}\}$. If there exists an edge $e_i$ satisfying \\  
(i) $d_G(e_i,e_1)\geq 6$, \\
(ii)  $\min \{d_G(e_i,e_j) \ | \ j = 2,3,\ldots,i-1\}  \geq 5$, and \\
(iii) we have equality in (i) or (ii) or both, \\
then add $e_i$ to $M$. 
We repeat this process until, after $k$ steps say, every edge not in $M_0\cup \{e_1\}$ is within distance $5$ of $e_1$, or within distance $4$ of an edge in $M_0$.  
Let $M=\{e_1, e_2,\ldots, e_k\}$. 

The sets $N_{\leq 3}(e_1)$ and $N_{\leq 2}(e_i)$ for $i=2,3,\ldots,k$ are pairwise 
disjoint. 
Let $T(e_1)$ be a spanning tree of $N_{\leq 3}(e_1)$ that 
contains $e_1$ and preserves the distances to $e_1$. 
For $i=2,3,\ldots, k$ let $T(e_i)$ be a spanning tree of $N_{\leq 2}(e_i)$ that 
contains $e_i$ and preserves the distances to $e_i$. 
Then the trees $T(e_i)$, $i=1,2,\ldots,k$, are vertex disjoint,
so the union $T(e_1) \cup \bigcup_{i=2}^k T(e_i)$ forms a subforest $T_1$ of $G$.  
It follows from the construction of $M$ that for every $i \in \{2,3,\ldots,k\}$
there exists an edge $f_i$ in $G$ joining a vertex in $T(e_i)$ to a vertex in $T(e_j)$
for some $j$ with $1 \leq j <i$.  Hence $T_2:=T_1 + \{f_2, f_3,\ldots, f_k\}$ is a subtree
of $G$. 
We extend $T_2$ to a spanning tree $T$ of $G$ that preserves the distances 
to a nearest vertex in $V(M)$.
In $T$ every vertex is within distance $6$ from some vertex in $V(M)$.  
Since the average eccentricity of any spanning tree of $G$ is not less than 
the average eccentricity of $G$, it suffices to show that  
\begin{equation} \label{eq:bound-on-avec-T-max-degree} 
{\rm avec}(T) \leq 
\frac{n-\Delta^*}{2\delta^*} \,  \frac{9n+3\Delta^*}{n} + 21.
\end{equation}
For every vertex $u \in V (T)$ let $u_M$ be a vertex in $V(M)$ closest to $u$ in $T$. 
We may assume that $u_M$ is a vertex incident with $e_i$ whenever $u\in V(T(e_i))$. 
As in the proof of Theorem \ref{theo:avec-mindegree-girth6} we define a
weight function $c: V(T) \rightarrow \mathbb{R}^{\geq 0}$ by 
\[ c(v) = | \{u \in V(T) \ | \ u_M = v\}|. \]
Now $d_T(x, x_M) \leq 6$ for all $x\in V(G)$. The same arguments as in the 
proof of Theorem \ref{theo:avec-mindegree-girth6} show that 
\begin{equation}  \label{eq:avec(T)-vs-avec_c(T)-maxdegree}
{\rm avec}(T) \leq {\rm avec}_c(T) + 6. 
\end{equation}
We consider the line graph $L$ of $T$ and define a new weight function 
$\overline{c}$ on $V(L) = E(T)$ by
\[ \overline{c}(uv) = \left\{ \begin{array}{cc}
 c(u) + c(v) & \textrm{if $uv\in M$,} \\
 0 & \textrm{if $uv \notin M$}. 
\end{array} \right. \]
As in the proof of Theorem \ref{theo:avec-mindegree-girth6}, we have
\begin{equation}  \label{eq:avec(T)-vs-avec(L)}
{\rm avec}_c(T) \leq {\rm avec}_{\overline{c}}(L) +1.
\end{equation}
Let $H$ be the graph obtained from $L^6[M]$ by joining $e_1$ to every edge
$e_i \in M$ for which $d_L(e_1,e_i)\leq 7$. 
Such edges $e_i$ exist since by the construction of $M$ we have 
$d_T(e_1, e_2)=6$ and thus $d_L(e_1, e_2) \leq 7$. 
Essentially the same argument as in the proof of Theorem 
\ref{theo:avec-mindegree-girth6} shows that $H$ is connected. \\
Let $e,f \in M$ and let $P$ be a shortest path from $e$ to $f$ in $H$ of length $\ell$ say. First assume that $P$ does not pass through $e_1$. Then each edge of $P$ yields a path in $L$ of length $6$, so $P$ yields a path from $e$ to $f$ of length at most $6\ell$. Now assume that $P$ passes through $e_1$. Then each edge on $P$ not incident with $e_1$ yields a path of length at most $6$ in $L$, while each edge of $P$ incident with $e_1$ yields a path of length at most $7$ in $L$. Since $P$ has at most two edges incident with $e_1$, $P$ yields a path of length at most $6\ell+2$. Hence
\begin{equation} \label{eq:avec(L)-vs-avec(H)} 
d_L(e,f) \leq 6d_{H}(e,f) +2. 
\end{equation} 
Now, for every edge $f \in E(T)$ there exists an edge $g \in M$ such that 
$d_L(f,g) \leq 6$. Hence 
$e_L(e) \leq 6 \, {\rm avec}_{\overline{c}}H(e) + 8$ for every $e \in  M$, and thus
\begin{equation}  \label{eq:avec(L)-vs-avec(H)}
{\rm avec}_{\overline{c}}(L) \leq 6 {\rm avec}_{\overline{c}}(H) + 8. 
\end{equation}
As in \eqref{eq:lower-bound-on-N2(e)-girth6} we have  
\begin{equation} \label{eq:lower-bound-weight-of-ei}
\overline{c}(e_2),  \overline{c}(e_3), \ldots, \overline{c}(e_k)  
                  \geq \delta^*. 
\end{equation}
By Lemma \ref{la:N3} we have 
\begin{equation} \label{eq:lower-bound-weight-of-e1}
\overline{c}(e_1) \geq |N_{\leq 3}(v_1)| \geq \Delta^*. 
\end{equation}
Since 
$\sum_{e\in M} \overline{c}(e)  = \sum_{v \in V(T)} c(v) = n$, 
we have
$|M| \leq \frac{n}{\delta^*}$. 
We now modify the weight function $\overline{c}$ to obtain a new weight
function $\overline{c}'$. We define  
\[ \overline{c}'(e_i) = \left\{ \begin{array}{cc} 
 \frac{\overline{c}(e_i) -\Delta^* + \delta^*}{\delta^*} & \textrm{if $i=1$,} \\[1mm]
 \frac{\overline{c}(e_i)}{\delta^*} & \textrm{if $i\geq 2$.} 
 \end{array} \right. \] 
Clearly, 
$\sum_{v\in V(H[M])} \overline{c}'(v) = \frac{n-\Delta^*+\delta^*}{\delta^*}=:N^*$. 
Hence 
\begin{eqnarray}   \label{eq:express-avec-as -linear-combination-1}
{\rm avec}_{\overline{c}'}(H) 
 & = & \frac{ EX_{\overline{c}'}(H)}{N^*} \nonumber \\
 & = & \frac{ \frac{1}{\delta^*} \Big[ \sum_{u\in M} \overline{c}(u) e_{H}(u)
        - (\Delta^* - \delta^*) e_{H}(e_1) \Big]}{N^*} \nonumber \\
 & = & 
  \frac{ EX_{\overline{c}}(H) - (\Delta^* - \delta^*)e_H(e_1)}{n-\Delta^* + \delta^*} 
                            \nonumber \\
 & = & \frac{n}{n-\Delta^* + \delta^*} {\rm avec}_{\overline{c}}(H) 
       - \frac{\Delta^* - \delta^*}{n-\Delta^* + \delta^*} e_H(e_1).
\end{eqnarray}
Rearranging yields
\begin{equation} \label{eq:express-avec-as -linear-combination-2}
{\rm avec}_{\overline{c}}(H) 
   =  \frac{n-\Delta^* + \delta^*}{n} {\rm avec}_{\overline{c}'}(H) 
       + \frac{\Delta^* - \delta^*}{n} e_H(e_1).
\end{equation}
We now bound the two terms on the right hand side of 
\eqref{eq:express-avec-as -linear-combination-2} separately. 
Note that $\overline{c}'(e_i) \geq 1$ for all $e_i \in M$.  Applying 
Lemma \ref{la:weighted-avec-maximise-by-path} we obtain 
\[ {\rm avec}_{\overline{c}'}(H) \leq {\rm avec}(P_{\lceil N^* \rceil}) 
       = \frac{3}{4} \lceil N^* \rceil - \frac{1}{2}. \]
Now $\lceil N^* \rceil  
     = \lceil \frac{n-\Delta^* + \delta^*}{\delta^*} \rceil 
     < \frac{n-\Delta^*}{\delta^*} + 2$. 
Hence 
\begin{equation}  \label{eq:bound-on-avec-H}
{\rm avec}_{\overline{c}'}(H) < \frac{3(n-\Delta^*)}{4\delta^*} + 1. 
\end{equation}   
To bound $e_H(e_1)$ note that $H$ has order $|M|$. Now 
$|M| = \sum_{e \in M} 1 \leq \sum_{e \in M} \overline{c}'(e_i) 
       = \frac{n-\Delta^*+\delta^*}{\delta^*}$. 
Hence 
\begin{equation}  \label{eq:bound-on-eccentricity-of-e1}
e_H(e_1) \leq |M|-1 = \frac{n-\Delta^*+\delta^*}{\delta^*}  - 1 
     = \frac{n-\Delta^*}{\delta^*}.
\end{equation} 
From \eqref{eq:express-avec-as -linear-combination-2}, \eqref{eq:bound-on-avec-H}
and \eqref{eq:bound-on-eccentricity-of-e1} we get, after some calculations, 
\begin{eqnarray}
{\rm avec}_{\overline{c}}(H) 
 & < & \frac{n-\Delta^* + \delta^*}{n} \Big( \frac{3(n-\Delta^*)}{4\delta^*} + 1 
    \Big) + \frac{\Delta^* - \delta^*}{n} \frac{n-\Delta^*}{\delta^*} \nonumber \\
  & = &  \frac{n-\Delta^*}{4\delta^*} \,  \frac{3n+\Delta^*}{n} 
         + \frac{3n-3\Delta^* + 4\delta^*}{4n} \nonumber \\
  & \leq  &  \frac{n-\Delta^*}{4\delta^*} \,  \frac{3n+\Delta^*}{n} + 1.         
       \label{eq:bound-on-avec-H-2}
\end{eqnarray}
Applying the inequalities \eqref{eq:avec(T)-vs-avec_c(T)-maxdegree}, \eqref{eq:avec(T)-vs-avec(L)},
\eqref{eq:avec(L)-vs-avec(H)} and  \eqref{eq:bound-on-avec-H-2}
we obtain
\begin{eqnarray*}
{\rm avec}(T) & \leq & {\rm avec}_{c}(T) + 6 \\ 
   & \leq & {\rm avec}_{\overline{c}}(L) + 7 \\
   & \leq & 6\, {\rm avec}_{\overline{c}}(H) + 15 \\
   & < & \frac{n-\Delta^*}{2\delta^*} \,  \frac{9n+3\Delta^*}{n} + 21, 
\end{eqnarray*}
as desired. 
\end{proof}

The following Theorem demonstrates that the bound in 
Theorem \ref{theo:avec-mindegree-maxdegree-girth6}
is sharp if $\delta-1$ is a prime power, except for an additive term $O(\sqrt{\Delta})$.

\begin{theorem}[\cite{AloDan-manu}]  \label{theo:cage-construction} 
Let $\delta, k \in \mathbb{N}$ such that $\delta-1$ is a prime power and $k\geq 7$. 
Then there exist a bipartite, $C_4$-free graph $G^{\delta,k}$ of minimum degree $\delta$,  
maximum degree $\Delta = \frac{(q^k-1)(q^{k-1}-1)}{(q^2-1)(q^2-q)} - \frac{1}{q}$,
where $q= \delta-1$, whose order $n_{\delta,k}$ satisfies   
\[ \Delta^* \leq 
n_{\delta,k} \leq  \Delta^* + 2 \sqrt{\Delta(\delta-2)} + \frac{1}{2}. \]
\end{theorem}

We make use of the fact that the graph in Theorem \ref{theo:cage-construction} has 
diameter at least $3$, which is easy to check from the construction (see \cite{AloDan-manu}). 
In the proof of Theorem \ref{theo:sharpness-example-min-max-degree} we make use of
a graph $G$, which was first described in \cite{AloDan-manu}.

\begin{theorem}   \label{theo:sharpness-example-min-max-degree}
Let $\delta\in \mathbb{N}$ be such that $\delta-1$ is a prime power.
Then there exist infinitely many connected graphs $G$ of minimum degree 
$\delta$ and girth $6$ with
\[ {\rm avec}(G) >  \frac{n-\Delta^*}{2\delta^*} \,  \frac{9n+3\Delta^*}{n} 
                      - O(\sqrt{\Delta}),  \]
where $\Delta$ is the maximum degree and $n$ the order of $G$.     
\end{theorem}

\begin{proof}
Let $q:=\delta-1$, so $q$ is a prime power. Let $k, \ell\in \mathbb{N}$ 
with $\ell$ even and $\ell$ sufficiently large. 
Consider the graph $G^{\delta,k}$ in Theorem \ref{theo:cage-construction} and let 
$u^1$ be a vertex of degree $\Delta$, and let $v^1$ be a vertex at distance 
three from $u^1$.

As in the construction of the graph $G_{\delta,\ell}$ in 
Theorem \ref{theo:sharpness-example-for-bip-C4-free-avec-mindegree} let 
$H^2, H^3,\ldots,H^{\ell}$ be isomorphic to $H_q$, but let
$H^1$ be the graph $G^{\delta,k}$. Denote the resulting graph by $G$. 
It is easy to verify that $G$ has minimum degree $\delta$ and 
maximum degree $\Delta$, and that its diameter is  $d(u_1,v_{\ell}) = 6\ell-3$. 
For the order $n$ of $G$ we have
\begin{equation} \label{eq:order-of-sharpness-example-for-max-degree}
n = n(G^{\delta,k}) + (\ell-1) n(H_{q})
     = n_{\delta,k} + (\ell-1) \delta^*. 
\end{equation}     
We now bound the average eccentricity of $G$ from below. 
For $i \in \{2,3,\ldots,\ell\}$ let $V(i)$ be the vertex set of $H^i$, 
and for $i=1$ let $V(1)$ be a set of $\delta^*$ vertices of $H^1$. 
Let $x \in V(i)$. If $i \leq \frac{1}{2}\ell$, then 
\[ e_G(x) \geq d_G(x, v^{\ell}) \geq d_G(v^i, v^{\ell}) = 6(\ell+1-i)-8, \]
and if $i > \frac{1}{2}\ell$, then
\[ e_G(x) \geq d_G(x, u^1) \geq d_G(u^i, u^1) = 6i-8. \]
The $(n_{\delta,k} - \delta^*$ vertices in $V(H^1)-V(1)$ have eccentricity at least
$6\ell -8$. Hence,   
\begin{eqnarray*}
EX(G) & = & 
  \big(\sum_{i=1}^{\ell/2} + \sum_{i=\ell/2+1}^{\ell} \big) \sum_{x \in V(i)} e_G(x)
    + \sum_{x \in V(H^1)-V(1)} e_G(x)  \\
 & \geq & 
  \Big(\sum_{i=1}^{\ell/2} \delta^* [ 6(\ell +1-i) - 8] \Big) 
  +  \Big( \sum_{i=\ell/2 + 1}^{\ell} \delta^* [6i- 8]   \Big)
   + (n_{\delta,k} - \delta^*) (6\ell-8)  \\
 & = &  ( \frac{9}{2} \ell^2 - 5 \ell) \delta^*
     + (n_{\delta,k} - \delta^*)(6\ell-8).
\end{eqnarray*}   
Now $\ell= \frac{n-n_{\delta,k}}{\delta^*} + 1$ by 
\eqref{eq:order-of-sharpness-example-for-max-degree}. Substituting this and 
dividing by $n$ yields, after simplification, 
\begin{eqnarray*} 
{\rm avec}(G) &\geq& \frac{n-n_{\delta,k}}{2\delta^*} \,  \frac{9n+3n_{\delta,k}}{n} 
                - 2 + \frac{3\delta^*}{2n} \\
     &>& \frac{n-n_{\delta,k}}{2\delta^*} \,  \frac{9n+3n_{\delta,k}}{n} 
                - 2. 
\end{eqnarray*}                
Now let $\varepsilon = n_{\delta,k} - \Delta^*$. Replacing $n_{\delta,k}$ by 
$\Delta^*+\varepsilon$ in the above lower bound, we obtain  
\begin{eqnarray*} 
{\rm avec}(G) 
     &>& \frac{n-\Delta^*-\varepsilon}{2\delta^*} \,  \frac{9n+3\Delta^*+3\varepsilon}{n} 
                - 2 \\  
     &=& \frac{n-\Delta^*}{2\delta^*} \,  \frac{9n+3\Delta^*}{n} 
         - \frac{\varepsilon}{2\delta^* n}(6n + 6 \Delta^* + 3\varepsilon) - 2.                  
\end{eqnarray*}    
Since $6n + 6 \Delta^* + 3\varepsilon \leq 12n$, and since 
$0 \leq \varepsilon \leq 2\sqrt{\Delta(\delta-2)} + \frac{1}{2}$
by Theorem \ref{theo:cage-construction} we have, 
for constant $\delta$ and large $n$ and $\Delta$, 
\[ {\rm avec}(G) >  \frac{n-\Delta^*}{2\delta^*} \,  \frac{9n+3\Delta^*}{n} 
                      - O(\sqrt{\Delta}), \]
as desired.                       
\end{proof}

Theorem \ref{theo:avec-mindegree-maxdegree-girth6} generalises 
Theorem \ref{theo:avec-mindegree-girth6} in the sense that it implies
(by setting $\Delta=\delta$) a bound that differs from Theorem
\ref{theo:avec-mindegree-girth6} only by having a weaker additive 
constant.

As in the previous section, a bound slightly weaker than that in 
Theorem \ref{theo:avec-mindegree-maxdegree-girth6} holds for all
$(C_4, C_5)$-free graphs. We omit the proof, which is very similar
to the proof of Theorem \ref{theo:avec-mindegree-maxdegree-girth6}.

\begin{theorem} \label{theo:avec-mindegree-maxdegree-C4-C5-free}
Let $G$ be a connected $(C_4, C_5)$-free graph of order $n$, minimum 
degree $\delta \geq 3$ and maximum degree $\Delta$. Then 
\[ {\rm avec}(G) \leq \frac{n-\Delta^\circ}{2\delta^\circ} \,  \frac{9n+3\Delta^\circ}{n} + 21, \]
where $\delta^\circ = 2\delta^2-5\delta+5$ if $\delta$ is even,  
$\delta^\circ = 2\delta^2-5\delta+7$ if $\delta$ is odd, and  
$\Delta^\circ= \Delta (\delta-1) + (\delta-2)\sqrt{\Delta(\delta-3)} +\frac{3}{2}$.  
\end{theorem}

We do not know if the bound in Theorem \ref{theo:avec-mindegree-maxdegree-C4-C5-free}
is sharp.

\end{document}